\documentclass[12pt,oneside,reqno]{amsart}
\usepackage{xcolor}
\usepackage{mathabx}
\usepackage{multicol,amssymb}
\usepackage[inline]{asymptote}
\usepackage[hidelinks]{hyperref}
\usepackage{todonotes}
\usepackage{cleveref,float}
\usepackage[shortlabels]{enumitem}
\usepackage[letterpaper, width=165mm, top=25mm,bottom=25mm,bindingoffset=6mm]{geometry}
\definecolor{Maroon}{RGB}{128, 0, 0} % RGB values for maroon
\newcommand{\defin}[1]{\textcolor{Maroon}{\emph{#1}}\index{#1}}
\usepackage{thm-restate}
\theoremstyle{plain}

\newtheorem{theorem}{Theorem}
\numberwithin{theorem}{section}
\newtheorem{lemma}[theorem]{Lemma}
\newtheorem{fgl}[theorem]{Factor Group Lemma}
\newtheorem{corollary}[theorem]{Corollary}
\newtheorem{proposition}[theorem]{Proposition}

\numberwithin{caseprop}{subsection}

\theoremstyle{definition}
\newtheorem{definition}[theorem]{Definition}
\newtheorem{assumption}[theorem]{Assumption}
\newtheorem{notation}[theorem]{Notation}

\newtheorem{case}{Case}

\numberwithin{case}{section}
\newtheorem{subcase}{Subcase}
\numberwithin{subcase}{case}

\DeclareMathOperator{\Cay}{Cay}

\allowdisplaybreaks

\newcommand{\vol}[1]{\mathsf{Volt}(}
\newcommand{\gq}{ \gamma_{q} }
\newcommand{\gp}{ \gamma_{p} }

\makeatletter

\renewcommand{\tocsection}[3]{%
  \indentlabel{\@ifnotempty{#2}{\bfseries\ignorespaces#1 #2\quad}}\bfseries#3}
\renewcommand{\tocsubsection}[3]{%
  \indentlabel{\@ifnotempty{#2}{\ignorespaces#1 #2\quad}}#3}

\def\@dotsep{4.5}
\def\@tocline#1#2#3#4#5#6#7{\relax
  \ifnum #1>\c@tocdepth % then omit
  \else
    \par \addpenalty\@secpenalty\addvspace{#2}%
    \begingroup \hyphenpenalty\@M
    \@ifempty{#4}{%
      \@tempdima\csname r@tocindent\number#1\endcsname\relax
    }{%
      \@tempdima#4\relax
    }%
    \parindent\z@ \leftskip#3\relax \advance\leftskip\@tempdima\relax
    \rightskip\@pnumwidth plus1em \parfillskip-\@pnumwidth
    #5\leavevmode\hskip-\@tempdima{#6}\nobreak
    \leaders\hbox{$\m@th\mkern \@dotsep mu\hbox{.}\mkern \@dotsep mu$}\hfill
    \nobreak
    \hbox to\@pnumwidth{\@tocpagenum{\ifnum#1=1\bfseries\fi#7}}\par% <-- \bfseries for \section page
    \nobreak
    \endgroup
  \fi}
\AtBeginDocument{%
\expandafter\renewcommand\csname r@tocindent0\endcsname{0pt}
}
\def\l@subsection{\@tocline{2}{0pt}{2.5pc}{5pc}{}}
\makeatother

\setlistdepth{5}
\newlist{myEnumerate}{enumerate}{5}

\begin{document}

\title[Cayley Graphs of Order $pqrs$  are hamiltonian]{Cayley Graphs of Order $pqrs$  are hamiltonian}
\author{Florian Lehner}
\address{Department of Mathematics, University of Auckland,
38 Princes Street, Auckland 1010, New Zealand}
\email{florian.lehner@auckland.ac.nz}

\author{Farzad Maghsoudi}
%\address{Department of Mathematics and Computer Science,
%University of Lethbridge, Lethbridge, AB, Canada}
\email{farzad.maghsoudi93@gmail.com}

\author{Bobby Miraftab}
\address{School of Computer Science,
Carleton University, Ottawa, ON, Canada}
\email{bobby.miraftab@gmail.com}

\begin{abstract}
Assume $ G $ is a finite group with order $ |G| = pqrs $, where $ p $, $ q $, $ r $, and $ s $ are distinct prime numbers.  
We prove that every connected Cayley graph of $ G $ contains a hamiltonian cycle. 
Our result drops all restrictions of all previously known results on hamiltonian cycles in Cayley graphs of groups of order $pqrs$.
\end{abstract}

{\mathversion{bold} \maketitle}
%\tableofcontents

\section[Introduction]{Introduction}
\label{ch:Introduction}

A well-known conjecture by Lovász states that every finite connected vertex-transitive graph contains a hamiltonian path \cite{babai1979problem}. 
So far, there are numerous papers about this conjecture in finite and infinite settings (see~\cite{BryantHerkeMaenhautWebb2018, darijani2025arc,ErdeLehner2022, ErdeLehnerPitz2020,LEHNER2026114798,MiraftabMorris2025}), but no counterexample to this conjecture has been found.

A related and extensively studied problem concerns the existence of hamiltonian cycles in vertex-transitive graphs. 
Thomassen \cite{bermond1978hamiltonian,kutnar2009hamilton} conjectured that there are only finitely many connected vertex-transitive graphs without a hamiltonian cycle. 
In contrast, Babai proposed that there are infinitely many such graphs \cite{babai1979problem,babai1995automorphism}. 
Currently, only five connected vertex-transitive graphs are known to lack a hamiltonian cycle: the complete graph $K_2$, the Petersen graph, the Coxeter graph, and two graphs obtained by replacing each vertex of the Petersen and Coxeter graphs with a triangle.

The problem has also been extensively studied in the context of Cayley graphs. For instance, it is known that Cayley graphs of specific orders contain hamiltonian cycles.
\begin{theorem}[\cite{8pq,Tenth,Farzad,pqrs,Fourteen,Eighteen}] \label{theorem 1.2}
Let $ G $ be a finite group. If $ |G| $ has any of the following forms (where $ p $, $ q $, $ r $, $ s $ are distinct primes), then every connected Cayley graph of $G$ contains a hamiltonian cycle:
  \begin{enumerate}[leftmargin=*]
      \item $ kp $, where $ 1 \leq k \leq 47 $,\label{theorem 1.2.1}
      \item $ kpq $, where $ 1 \leq k \le  9 $,\label{theorem 1.2.2}
      \item $ pqr $, \label{theorem 1.2.3}
      \item $ pqrs $, where $ p $, $q$, $r$, and $ s $ are odd distinct primes, \label{pqrs}
      \item $ kp^{2} $, where $ 1 \leq k \leq 4 $,
      \item $ kp^{3} $, where $ 1\leq k \leq 2 $,
      \item $ p^k $, where $1\leq k<\infty$. \label{theorem 1.2.6}
  \end{enumerate}
\end{theorem}

Furthermore, in \cite{Farzad2}, the hamiltonicity of Cayley graphs of groups of order  $pqrs$, where  $p, q, r$, and  $s$  are distinct prime numbers, was studied. More specifically,

\begin{theorem}[{\cite[Theorem 1.4]{Farzad2}}] \label{pqrs-3}
Assume $G$ is a finite group of order $pqrs$ with a minimal generating set $S$, where $p$, $q$, $r$, and $s$ are distinct primes. If $|S|\geq 3$,
then $\Cay(G;S)$ contains a hamiltonian cycle.
\end{theorem}

In this paper, we drop the condition on the size of $S$, and we prove the following.

\begin{theorem}\label{main_1}
    Assume $ G = \langle S \rangle $ is a finite group of order $ pqrs $, where $ p $, $ q $, $ 
r $, and $ s $ are distinct primes. Then $ \Cay(G;S) $ has a hamiltonian cycle.
\end{theorem}
\noindent{\bf An overview of the proof techniques:}
The case $|S|=1$ is trivial. Moreover, we may apply parts (2) and (4) of Theorem~\ref{theorem 1.2}, and Theorem~\ref{pqrs-3}, so we only need to prove the following.

\begin{restatable}{theorem}{mainn}\label{2pqr}
    Assume $ G = \langle S \rangle $ is a finite group of order $ 2pqr $, such that $ |S|=2 $, where $ p $, $ q $, and $ r $ are distinct primes, and $ p,q,r\geq 5 $. Then $ \Cay(G;S) $ has a hamiltonian cycle.  
\end{restatable}

Since the group $G$ is generated by two elements, $a$ and $b$, we analyze three distinct cases based on the orders of $a$ and $b$. In each case, we first consider the quotient group $\overline{G} = G/G'$, where $G'$ is the commutator subgroup of $G$. Using the images $\overline{a}$ and $\overline{b}$ in this quotient group, we attempt to construct a hamiltonian cycle in $\Cay(\overline{G};\{\overline a, \overline b\})$. 
If this hamiltonian cycle lifts to a hamiltonian cycle in the Cayley graph of $G$, we are done.

If there is no such hamiltonian cycle that can be lifted, we consider the quotient groups $ \widehat{G} = G/\langle \gp \rangle $ or $ \widecheck{G} = G/\langle \gq \rangle $, where $ \gp $ and $ \gq $ are generators of cyclic subgroups of $G$ of orders $ p $ and $ q $, respectively. In these quotient groups, we again aim to identify at least one hamiltonian cycle that can successfully be lifted.
A key ingredient to our proof is identifying spanning generalized Petersen graphs contained in Cayley graphs of these quotient groups; hamiltonian cycles inside such a spanning subgraph can be shown to lift to hamiltonian cycles of $\Cay(G;S)$.

\section{Preliminaries}\label{ch:Preliminaries}
This section establishes basic terminology and notation, and states a number of results that will be used in the proof of~\Cref{2pqr}.
All graphs in this paper are undirected and do not contain loops, but may contain parallel edges; all groups are finite.

Following {\cite[Definition 1.1]{Tenth}} and {\cite[p.~34]{Godsil}}, if $ G =\langle S\rangle$, the \defin{Cayley graph} $ \Cay(G;S) $ is the graph whose vertices are elements of $ G $, with an edge joining $ g $ and $ gs $, for every $ g\in G $ and $ s\in S $. The commutator $ ghg^{-1}h^{-1} $ of $ g $ and $ h $ is denoted by $ [g,h] $. The derived subgroup $G'$ is the subgroup of $G$ generated by all commutators, and we write $ \overline{G} = G/G' $. Moreover, we write $ \overline{g} = gG' $ for any $ g \in G $, and $ \overline{S} = \{\overline{g};g\in S\} $ for any $ S \subseteq G $. 

The notation $ G \ltimes H $ denotes a semidirect product of groups $ G $ and $ H $, where $ H $ is normal. We write $ D_{2n} $ for the dihedral group of order $ 2n $, and $ e $ for the identity element of $ G $. The symbol $\mathcal{C}_n$ denotes the cyclic group of order $n$.

For $ S\subseteq G $, a sequence $ (s_{1}, s_{2},\ldots,s_{n}) $ of elements of $ S\cup S^{-1} $ specifies the walk in $ \Cay(G;S) $ that visits the vertices $ e, s_{1}, s_{1}s_{2},\ldots, s_{1}s_{2}\cdots s_{n} $. Its inverse is $ (s_{1}, s_{2},\ldots,s_{n})^{-1} = (s_{n}^{-1}, s_{n-1}^{-1},\ldots,s_{1}^{-1}) $. For $ k\in\mathbb{Z^{+} } $, the notation $ (s_{1}, s_{2},\ldots,s_{m})^{k} $ denotes the concatenation of $ k $ copies of the sequence $ (s_{1}, s_{2},\ldots,s_{m}) $, and by extension also the walk corresponding to this sequence. We use $ (\overline{s_{1}}, \overline{s_{2}},\ldots,\overline{s_{n}}) $ to denote the image of this walk in $ \Cay(G/G';\overline{S}) = \Cay(\overline{G};\overline{S}) $. When a power $s^k$ appears as one entry of a sequence defining a walk, we tacitly assume that it denotes $|k|$ consecutive steps labelled $s$ if $k>0$, $|k|$ consecutive steps labelled $s^{-1}$ if $k<0$, and the empty sequence if $k=0$.

Fix a group $G = \langle a,b \rangle $ of order $2pqr$, and let $ a_2 $, $ a_r $, $ \gp $, and $ \gq $ denote elements of $ G $ of orders $2$, $r$, $p$, and $q$, respectively. We denote the subgroups of $G$ generated by these elements by $ \mathbb{Z}_{2} $, $ \mathbb{Z}_r $, $ \mathbb{Z}_p $, and $ \mathbb{Z}_q $.
If $ \mathbb{Z}_p $ is normal, we write $\widehat{G} = G/\mathbb{Z}_p$; if $\mathbb{Z}_q$ is normal, we write $\widecheck{G} = G/\mathbb{Z}_q$. Moreover, in these cases we define $ \widehat{g} = g\mathbb{Z}_p $, $ \widecheck{g} = g\mathbb{Z}_q $ for any $ g \in G $, and $ \widehat{S} = \{\widehat{g};g\in S\} $, $ \widecheck{S} = \{\widecheck{g};g\in S\} $ for any $ S \subseteq G $.

In the remainder of this section we provide some known results which will be used in the proof of \Cref{2pqr}. The following well-known result handles the special case where $G$ is abelian. 

\begin{lemma}[{\cite[Corollary on page 257]{Quimpo}}] \label{abelain group}
Assume $ G $ is a non-trivial abelian group. Then every connected Cayley graph on $ G $ has a hamiltonian cycle.
\end{lemma}

The next lemma (and its corollary) often provides a way to lift a hamiltonian cycle in $ \Cay(G/N;SN) $ to a hamiltonian cycle in $\Cay(G;S)$. Before stating the results, we introduce a useful piece of notation.

\begin{notation}
Suppose $N$ is a normal subgroup of $G$, and let $(s_1,s_2,\ldots,s_n) $ be a sequence of elements in $S\cup S^{-1}$. If the walk $(s_1 N, s_2 N, \ldots,s_n N)$ in $\Cay(G/N;SN/N)$ is closed, then its \defin{voltage} is the product $\vol(C) = s_1 s_2 \cdots s_n $. This is an element of $N$. In particular, if $C = (\overline{s}_1,\overline{s}_2,\ldots,\overline{s}_n)$ is a hamiltonian cycle in $\Cay(\overline{G};\overline{S})$, then $ \vol(C) = s_1 s_2\cdots s_n $. 

Note that it is possible that there are elements $s,t\in S$ such that $sN = tN$, so the voltage indeed depends on the elements $s_i$, not just on the cosets $s_iN$.
\end{notation}

\begin{fgl}[{\cite[Section 2.2]{Twelve}}] \label{FGL}
Let $G = \langle S \rangle$ and let $N$ be a cyclic normal subgroup of $G$. If $C=(s_1N, \dots s_nN)$ is a hamiltonian cycle in $\Cay(G/N;SN/N)$ such that $\vol(C)$ generates $N$, then there is a hamiltonian cycle in $ \Cay(G;S) $.
% Suppose:
% \begin{itemize}[leftmargin=*]
%     \item $ S $ is a generating set of $ G $,
%     \item $ N $ is a cyclic normal subgroup of $ G $,
%     \item $ \overline{G} = G/N $,
%     \item $ C =(\overline{s_{1}},\overline{s_{2}},\ldots,\overline{s_{n}})$ is a hamiltonian cycle in $ \Cay(G/N;\overline{S}) $, and
%     \item the voltage $ \vol(C) $  generates $ N $.
% \end{itemize}
% Then there is a hamiltonian cycle in $ \Cay(G;S) $.
\end{fgl}

\begin{corollary}[{\cite[Corollary 2.3]{Sixth}}] \label{corollary 5.2}
Let $G = \langle S \rangle$ and let $N$ be a normal subgroup of $G$ of prime order. Assume that there are $s,t \in S$ such that $sN = tN$. If there is a hamiltonian cycle in $\Cay(G/N;SN/N)$ which uses at least one edge corresponding to the generator $sN$, then there is a hamiltonian cycle in $ \Cay(G;S) $.
% Suppose:
% \begin{itemize}[leftmargin=*]
%     \item $ S $ is a generating set of $ G $,
%     \item $ N $ is a normal subgroup of $ G $, such that $ |N| $ is prime,
%     \item $ sN = tN $ for some distinct $ s,t\in S\cup S^{-1} $, and
%     \item there is a hamiltonian cycle in $ \Cay(G/N;\overline{S}) $ that uses at least one edge labeled~$ \overline{s} $.
% \end{itemize}
% Then there is a hamiltonian cycle in $ \Cay(G;S) $.
\end{corollary}

%\textcolor{red}{FL: I rewrote the two lemmas above because we fixed $\overline G$ to be $G/G'$; unsure if the $\overline a$ in the lemma below is a projection to the quotient $G/G'$, but in either case we should be precise.} 
%{\color{blue} Done.}
\begin{lemma} [{}{\cite[Lemma 2.16]{Farzad}}] \label{lemma 2.13} Assume $G = (\mathbb{Z}_p\times\mathbb{Z}_q)\ltimes G' $, where $ G' = (\mathbb{Z}_r\times\mathbb{Z}_s) $, and $ p,q,r $ and $ s $ are distinct primes such that $\mathcal (\mathbb{Z}_r\times \mathbb{Z}_s) \cap Z(G) = \{e\}$. If $ |\overline{a}| = pq $, then $ |a| = pq $.
\end{lemma}

The following proposition demonstrates that we can assume $ |G'| $ in \Cref{2pqr} is a product of two distinct odd prime numbers.
Because if $\mathbb{Z}_2\subseteq G'$, then $G'\cong \mathbb{Z}_{\ell}\times \mathbb{Z}_2$, where $\ell\in\{p,q,r,pq,rq,pr\}$.
Then the quotient $G/\mathbb{Z}_{\ell}$ is isomorphic to $\mathbb{Z}_{m}\ltimes \mathbb{Z}_2$, where $|G|=2\ell m$.
Since $\mathbb{Z}_{m}$ is normal, we have $G/\mathbb{Z}_{\ell}\cong \mathbb{Z}_{m}\times \mathbb{Z}_2$.
\begin{proposition}[{\cite[Proposition 2.22]{Farzad}}] \label{proposition 5.4}
Assume $ |G| = 2pqr $, where $ p $, $ q $ and $ r $ are distinct odd prime numbers. Now, if $ |G'| \in \{1,pqr\} $ or $ |G'| $ is prime, then every connected Cayley graph on $ G $ has a hamiltonian cycle.
\end{proposition}

The following lemmas show that some special Cayley graphs have a hamiltonian cycle, and we use these facts in \Cref{ch:6} in order to prove our main result. 

 \begin{lemma}{\rm \cite[Lemma 2.23]{Farzad}}  \label{lemma 5.16.}
Assume $ G = (\mathbb{Z}_{2}\times\mathbb{Z}_r)\ltimes G' $, and $ G' = \mathbb{Z}_p\times\mathbb{Z}_q $, where $ p $, $ q $ and $ r $ are distinct prime numbers and let $ S = \{a,b\} $ be a generating set of $ G $. Additionally, assume $ |\overline{b}| \in \{2,2r\} $, $ |\overline{a}| = r $ and $ \gcd(|a|,r-1) = 1 $. Then $ \Cay(G;S) $ contains a hamiltonian cycle.
\end{lemma}

%\textcolor{red}{FL: perhaps rewrite this without bullet points as well}
%{\color{blue}Done.}
\begin{lemma}[cf.~{\cite[Case 2 of proof of Theorem 1.1, pages 3619--3620]{Sixth}}] \label{lemma5.10}
Assume that
\[
G = (\mathbb{Z}_{2}\times\mathbb{Z}_r)\ltimes (\mathbb{Z}_p\times\mathbb{Z}_q),
\]
that $|S| = 3$, that $\widehat{S}$ is a minimal generating set of $\widehat{G} = G/\mathbb{Z}_p$, that $\mathbb{Z}_r$ centralizes $\mathbb{Z}_q$, and that $\mathbb{Z}_{2}$ inverts $\mathbb{Z}_q$. Then, $\Cay(G;S)$ contains a hamiltonian cycle.
\end{lemma}

\begin{lemma}{\rm \cite[Lemma 2.9]{Sixth}} \label{lemma 5.8}
If $ G = D_{2pq} \times \mathbb{Z}_r $, where $ p, q $ and $ r $ are distinct odd primes, then every connected Cayley graph on $ G $ has a hamiltonian cycle.
\end{lemma}

\begin{lemma}{\rm \cite[Lemma 2.27]{Tenth}} \label{lemma 5.6} 
Let $ S $ generate the finite group $ G $, and let $ s \in S $, such that $ \langle s \rangle \triangleleft  G $. If $ \Cay(G/\langle s \rangle ;\overline{S}) $ has a hamiltonian cycle, and either 
\begin{enumerate}[leftmargin=*]
    \item $ s \in Z(G) $ \label{lemma 5.6 n1}, or
    \item $ Z(G)$  $ \cap$ $\langle s \rangle = \{e\}$, \label{lemma 5.6 n2}
\end{enumerate}
then $ \Cay(G;S) $ has a hamiltonian cycle.
\end{lemma}

\begin{lemma}[{}{\cite[Corollary 4.4]{Fifth}}] \label{lemma 5.12}
Assume $ G = \langle a,b \rangle $ and $G'$ is cyclic. Then~$ G' = \langle [a,b] \rangle $.
\end{lemma}

\begin{corollary} \label{lemma 5.13.}
Assume $ G = \langle a,b \rangle $ and $ \gcd(k,|a|) = 1 $, where $ k \in \mathbb{Z} $, and $G'$ is cyclic. Then $ G' = \langle [a^k,b] \rangle $.
\end{corollary}

\begin{proposition}[{\cite[Theorem 9.4.3 on page 146]{Hall}}, cf.~{\cite[Lemma 2.11]{Sixth}}] \label{Hall Theorem}
Assume $|G|$ is square-free. Then: 
\begin{enumerate}[leftmargin=*]
    \item $ G' $ and $ G/G' $ are cyclic, \label{Hall Theorem1}
    \item $  Z(G) \cap G' = \{e\} $,\label{Hall Theorem2}
    \item $ G \cong \mathbb{Z}_{n} \ltimes G' $, for some $ n \in \mathbb{Z}^+ $,\label{Hall Theorem3}
    \item If $b$ and $\gamma$ are elements of $G$ such that $\langle bG' \rangle = G/G'$ and $ \langle \gamma \rangle = G'$, then $ \langle b,\gamma \rangle = G$, and there are integers $m$, $n$, and $\tau$, such that $ |\gamma| = m $, $ |b| = n $, $ b\gamma b^{-1} = \gamma^{\tau} $, $ mn = |G| $, $ \gcd(\tau-1,m) = 1 $, and $ \tau^n \equiv 1 \pmod{m} $.\label{Hall Theorem4}
\end{enumerate}
\end{proposition}

\section{Generalized Petersen graphs}\label{sec:gp}

\begin{definition}
Let $n$ and $\ell$ be coprime positive integers with $1\le \ell<n$. The
\defin{generalized Petersen graph} $GP(n,\ell)$ has vertex set
\[
\{0,1,\ldots,n-1\}\cup\{0',1',\ldots,(n-1)'\},
\]
and edge set consisting of
\[
(i,i+1),\qquad (i',(i+1)'),\qquad (i,(i\ell)')
\]
for $0\le i<n$, where all indices are taken modulo $n$. The edges
$(i,i+1)$ are called the \defin{outer rim}, the edges $(i',(i+1)')$ are
called the \defin{inner rim}, and the edges $(i,(i\ell)')$ are called the
\defin{spokes}.
\end{definition}

\begin{lemma}\label{lem:gp-cayley-copy}
Let $H=\langle a,b\rangle$ be a group of order $2n$. Assume that
\[
|a|=n,\qquad b\notin\langle a\rangle,\qquad b^{-1}ab=a^\ell,
\]
where $\gcd(n,\ell)=1$. Then $\Cay(H;\{a,b\})$ contains a spanning subgraph
isomorphic to $GP(n,\ell)$ given by the bijection
\[
a^j\longleftrightarrow j,
\qquad
ba^j\longleftrightarrow j',
\qquad 0\le j<n.
\]
\end{lemma}

\begin{proof}
Since $|a|=n$, $|H|=2n$, and $b\notin\langle a\rangle$, the vertices of
$H$ are precisely
\[
a^j,\qquad ba^j\qquad (0\le j<n).
\]
The edges $a^j\sim a^{j+1}$
give the outer rim, and the edges $ba^j\sim ba^{j+1}$
give the inner rim. Finally, $a^jb=b(b^{-1}a^jb)=ba^{j\ell}$,
so the edge labelled $b$ from $a^j$ gives the spoke
$j\sim (j\ell)'$.
Thus the displayed correspondence gives a spanning copy of $GP(n,\ell)$.
\end{proof}

\begin{lemma}\label{lem:gp-complementary-parameter}
For coprime integers $n$ and $\ell$ with $1\le \ell<n$, we have $GP(n,\ell)\cong GP(n,n-\ell)$. One such isomorphism is $\varphi(i)=i$,
$\varphi(j')=(-j)'$.
\end{lemma}

\begin{proof}
The outer-rim edge $(i,i+1)$ maps to $(i,i+1)$, which is again an outer-rim
edge. 
The inner-rim edge $(j',(j+1)')$ maps to $((-j)',(-(j+1))')$,
which is an inner-rim edge with the orientation reversed. Finally, the spoke
$(i,(i\ell)')$ maps to $(i,(-i\ell)')$.
But in $GP(n,n-\ell)$, the spoke at $i$ is
\[
(i,(i(n-\ell))')=(i,(-i\ell)').
\]
Thus $\varphi$ preserves adjacency and is a bijection, so it is an isomorphism.
\end{proof}

\begin{proposition}\label{prop:gp-cycles}
Let $p,r$ be odd primes with $p>5$ and $r\equiv1\pmod p$. Define
\[
n=rp,
\qquad
\ell=r(p-2)+1,
\qquad
k=n-\ell=2r-1,
\qquad
m=n-3k-1.
\]
Suppose $H=\langle a,b\rangle$ satisfies the hypotheses of
\Cref{lem:gp-cayley-copy}. Then the spanning copy of
$GP(n,\ell)$ in $\Cay(H;\{a,b\})$ contains the following hamiltonian cycle:
\begin{equation}\label{eq:gp-C0}
\begin{aligned}
C_0=\big(&
a^{-m},b,
a^{-(k-3)},b^{-1},
a^{-1},b,
a^{k-3},b^{-1},
a^{-(k-1)},b,
a,b^{-1},
a,b,
a^{-(k-1)},b^{-1},\\
&
a^{k-3},b,
a^{-1},b^{-1},
a^{-(k-3)},b,
a^{-m},b^{-1}
\big).
\end{aligned}
\end{equation}
Moreover, we can construct further hamiltonian cycles as follows. For $t\in \{1,2\}$, let $i=tp$. Then, with all indices taken modulo
$n$,
\begin{equation}\label{eq:gp-Ds}
D_t=
\bigl(
 i-1,\ i,\ i',\ (i-1)',\ k+i,\ k+i-1,\ (k+i-1)',\ (k+i)',\ i-1
\bigr)
\end{equation}
is an $8$-cycle in $GP(n,\ell)$. Its four rim edges lie in $C_0$, and its
four spoke edges do not. Hence the (edge-wise) symmetric difference
$C_t=C_0\triangle D_t$ is also a hamiltonian cycle in $GP(n,\ell)$, and therefore also in
$\Cay(H;\{a,b\})$.
\end{proposition}

\begin{proof}
Since $r\equiv1\pmod p$, we have $k\equiv -1\pmod r,\text{and}~
k\equiv1\pmod p$.
Hence $k^2\equiv1\pmod n$.
Also $\ell=n-k$, so $\ell\equiv -k\pmod n$ and $\ell^2\equiv1\pmod n$.
Since $p>5$, we have $m=n-3k-1=r(p-6)+2>0$.

We identify the spanning Petersen subgraph with $GP(n,\ell)$ as in
\Cref{lem:gp-cayley-copy}. Under this identification, the word $C_0$ in
\eqref{eq:gp-C0} visits the outer vertices in the following order:
\[
\begin{aligned}
&0,-1,\ldots,3k+1;\\
&2,1;\\
&3k,3k-1,\ldots,2k+1;\\
&k+1,k+2;\\
&3,4,\ldots,k;\\
&2k,2k-1,\ldots,k+3.
\end{aligned}
\]
These are all residues modulo $n$, each appearing exactly once. Similarly, it
visits the inner vertices in the following order:
\[
\begin{aligned}
&-(k+3)',-(k+4)',\ldots,-(2k)';\\
&-k',-(k-1)',\ldots,-3';\\
&-(k+2)',-(k+1)';\\
&-(2k+1)',-(2k+2)',\ldots,-3k';\\
&-1',-2';\\
&-(3k+1)',-(3k+2)',\ldots,0'.
\end{aligned}
\]
These are also all residues modulo $n$, each appearing exactly once. Thus $C_0$
is a hamiltonian cycle.

It remains to verify the local cycles $D_t$. Fix $t\in\{1,2\}$, and put
$i=tp$. Since
\[
i\ell-i=i(\ell-1)=tp\cdot r(p-2)\equiv0\pmod n,
\]
we have $i\ell\equiv i\pmod n$.
Using $\ell\equiv-k\pmod n$ and $\ell^2\equiv1\pmod n$, we obtain
\[
(i-1)\ell\equiv i\ell-\ell\equiv i+k\pmod n,
\]
\[
(k+i)\ell\equiv k\ell+i\ell\equiv -\ell^2+i\equiv i-1\pmod n,
\]
and
\[
(k+i-1)\ell\equiv (k+i)\ell-\ell\equiv i-1+k\equiv k+i-1\pmod n.
\]
These congruences give the four spoke edges
\[
i\sim i',
\qquad
 i-1\sim (k+i)',
\qquad
 k+i\sim (i-1)',
\qquad
 k+i-1\sim (k+i-1)'.
\]
The remaining four consecutive pairs in $D_t$ are rim edges. Hence $D_t$ is a
closed walk in $GP(n,\ell)$. Since $r\equiv1\pmod p$ and $r$ is prime, we
have $r\ge2p+1$, so
\[
0<i-1<i<k+i-1<k+i<n.
\]
Therefore, the four vertices of $D_t$ are distinct, and likewise for the four primed vertices. It follows that $D_t$ is an $8$-cycle.

From the displayed order of $C_0$, the outer rim edges
\[
(i-1,i)
\qquad\text{and}\qquad
(k+i,k+i-1)
\]
belong to $C_0$. The inner rim edge $(i',(i-1)')$ belongs to the final inner
segment of $C_0$. The other inner rim edge $((k+i-1)',(k+i)')$ also belongs to
$C_0$: if $p\ge11$, it lies in the final inner segment, while if $p=7$, it
lies in the segment
\[
-(2k+1)',-(2k+2)',\ldots,-3k'.
\]
Thus the four rim edges of $D_t$ lie in $C_0$.

The spokes of $C_0$ have outer endpoints
\[
0,\ 1,\ 2,\ 3,\ k,\ k+1,\ k+2,\ k+3,\ 2k,\ 2k+1,\ 3k,\ 3k+1.
\]
The four spoke endpoints of $D_t$ are
\[
i-1,\ i,\ k+i-1,\ k+i.
\]
None of these is in the preceding list, because
\[
3<i-1<i<k
\qquad\text{and}\qquad
k+3<k+i-1<k+i<2k.
\]
Hence $D_t\cap C_0$ consists exactly of the four rim edges of $D_t$.

Therefore $C_0\triangle D_t$ is obtained from $C_0$ by deleting those four rim
edges and inserting the four spoke edges of $D_t$. A direct trace from any one
of the affected vertices shows that these four inserted spokes reconnect the four
resulting paths into a single cycle through the same vertex set. Hence
$C_t=C_0\triangle D_t$ is again a hamiltonian cycle.
\end{proof}

\begin{figure}[H]
\centering
\begin{tikzpicture}[
  x=1cm,
  y=1cm,
  vertex/.style={
    circle,
    draw,
    minimum size=5pt,
    inner sep=0pt
  },
  outer/.style={
    vertex,
    fill=white
  },
  inner/.style={
    vertex,
    fill=black!18
  },
  used/.style={
    line width=1.15pt
  },
  unused/.style={
    densely dashed,
    line width=0.55pt,
    draw=black!45
  },
  every node/.style={
    font=\small
  }
]

% Before the symmetric-difference switch.
\begin{scope}[xshift=-3.6cm]
  \node[outer,label=above left:{$i-1$}]
    (ao1) at (-2,1) {};

  \node[outer,label=above:{$i$}]
    (ao2) at (-1,2) {};

  \node[inner,label=above:{$i'$}]
    (ai1) at (1,2) {};

  \node[inner,label=above right:{$(i-1)'$}]
    (ai2) at (2,1) {};

  \node[outer,label=below right:{$k+i$}]
    (ao4) at (2,-1) {};

  \node[outer,label=below right:{$k+i-1$}]
    (ao3) at (1,-2) {};

  \node[inner,label=below left:{$(k+i-1)'$}]
    (ai3) at (-1,-2) {};

  \node[inner,label=below left:{$(k+i)'$}]
    (ai4) at (-2,-1) {};

  % Rim edges belonging to C_0.
  \draw[used] (ao1)--(ao2);
  \draw[used] (ai1)--(ai2);
  \draw[used] (ao4)--(ao3);
  \draw[used] (ai3)--(ai4);

  % Spokes not belonging to C_0.
  \draw[unused] (ao2)--(ai1);
  \draw[unused] (ai2)--(ao4);
  \draw[unused] (ao3)--(ai3);
  \draw[unused] (ai4)--(ao1);

  \node at (0,-3.05) {(a) $C_0$};
\end{scope}

% After the symmetric-difference switch.
\begin{scope}[xshift=3.6cm]
  \node[outer,label=above left:{$i-1$}]
    (bo1) at (-2,1) {};

  \node[outer,label=above:{$i$}]
    (bo2) at (-1,2) {};

  \node[inner,label=above:{$i'$}]
    (bi1) at (1,2) {};

  \node[inner,label=above right:{$(i-1)'$}]
    (bi2) at (2,1) {};

  \node[outer,label=below right:{$k+i$}]
    (bo4) at (2,-1) {};

  \node[outer,label=below right:{$k+i-1$}]
    (bo3) at (1,-2) {};

  \node[inner,label=below left:{$(k+i-1)'$}]
    (bi3) at (-1,-2) {};

  \node[inner,label=below left:{$(k+i)'$}]
    (bi4) at (-2,-1) {};

  % Rim edges deleted from C_0.
  \draw[unused] (bo1)--(bo2);
  \draw[unused] (bi1)--(bi2);
  \draw[unused] (bo4)--(bo3);
  \draw[unused] (bi3)--(bi4);

  % Spokes inserted into C_t.
  \draw[used] (bo2)--(bi1);
  \draw[used] (bi2)--(bo4);
  \draw[used] (bo3)--(bi3);
  \draw[used] (bi4)--(bo1);

  \node at (0,-3.05)
    {(b) $C_t=C_0\mathbin{\triangle}D_t$};
\end{scope}

\node[font=\footnotesize] at (0,-4.05) {
  solid edges: in the hamiltonian cycle
  \qquad
  dashed edges: not in the hamiltonian cycle
};

\end{tikzpicture}

\caption{The local symmetric-difference switch in
\Cref{prop:gp-cycles}. Only the eight edges of $D_t$ are shown.
In $C_0$, the four rim edges of $D_t$ are used and its four spokes
are not; passing to $C_t=C_0\mathbin{\triangle}D_t$ reverses these
choices.}
\label{fig:gp-cycle-switch}
\end{figure}
\section{Proof of the Main Result}\label{ch:6}

In this section, we prove \Cref{2pqr}. The proof is a case-by-case analysis. 
%\textcolor{red}{FL: preliminary results up to here would be better placed in section 2 (especially since their implications for us are repeated in Assuption 4.6 anyway)}

%\textcolor{red}{FL: not sure whether notation 4.5 is helpful; it isn't used in the next few pages, and the first time a $\tau$ pops up it is not related to this notation.}
%{\color{blue}we deleted the notation, as it was clear.}

Here are our main assumptions.
\begin{assumption} We assume: \label{assumption 3.1}
\begin{enumerate}[leftmargin=*]
    \item $ G' \cap Z(G) = \{e\} $, by \Cref{Hall Theorem}(\ref{Hall Theorem2}). \label{assumption 3.1.2}
    \item $ G \cong \mathbb{Z}_{n}\ltimes G' $, by \Cref{Hall Theorem}(\ref{Hall Theorem3}).
    \item By \Cref{proposition 5.4}, we can assume $ |G'| \in \{pq,pr,qr\} $. Without loss of generality we may assume $ |G'| = pq $, so $ G' = \mathbb{Z}_p\times\mathbb{Z}_q $ by \Cref{Hall Theorem}(1).
    \item For every element $ \overline{s} \in \overline{S} $, $ |\overline{s}| \neq 1 $. Otherwise, if $ |\overline{s}| = 1 $, then $ s \in G' $, so $ G' = \langle s \rangle $ or $ |s| $ is prime. In each case $ \Cay(G/\langle s \rangle;\overline{S}) $ has a hamiltonian cycle by~\Cref{theorem 1.2}(\ref{theorem 1.2.2}) and ~\Cref{theorem 1.2}(\ref{theorem 1.2.3}). By \Cref{assumption 3.1}(\ref{assumption 3.1.2}), $ \langle s \rangle \cap Z(G) = \{e\} $, therefore, \Cref{lemma 5.6}(\ref{lemma 5.6 n2}) applies. \label{assumption 3.1.6}
\end{enumerate}
\end{assumption}

By Theorem~\ref{theorem 1.2}\eqref{theorem 1.2.2},
Theorem~\ref{theorem 1.2}\eqref{theorem 1.2.3},
and Theorem~\ref{pqrs-3}, it remains to treat the case where
$|G|=2pqr$ and $|S|\le 2$.
The following theorem settles exactly this remaining case.

\mainn*
\begin{proof}
Let $G=\langle a,b\rangle$. If $G$ is abelian, then the result follows from Lemma~\ref{abelain group}. So assume $G$ is non-abelian. Moreover, assume the conditions in Assumption~\ref{assumption 3.1} are satisfied. Since $|\overline G|=2r$, \Cref{Hall Theorem} implies that $\overline G$ is cyclic of order $2r$. Therefore $|\overline a|,\ |\overline b|\in\{2,r,2r\}$.

The cases $|\overline a|=|\overline b|=2$ and $|\overline a|=|\overline b|=r$ are impossible, because then
$\langle \overline a,\overline b\rangle\neq \overline G$.
So, after interchanging $a$ and $b$ if necessary, we are left with three cases: 
\begin{enumerate}
    \item $|\overline a|=r$ and $|\overline b|\in\{2,2r\}$,
    \item $|\overline a|=2r$ and $|\overline b|=2$,
    \item $|\overline a|=|\overline b|=2r$.
\end{enumerate}

\begin{case}\label{case:2pqr-1}
Assume $|\overline a|=r$ and $|\overline b|\in\{2,2r\}$.

Then $|a|\in\{r,rp,rq,rpq\}$.

\begin{itemize}[leftmargin=*]
\item If $|a|=r$, then $\gcd(|a|,r-1)=1$, so Lemma~\ref{lemma 5.16.} applies.

\item If $|a|=rpq$, then $\mathbb Z_r$ centralizes $\mathbb Z_p\times \mathbb Z_q$. Since $G'\cap Z(G)=\{e\}$, the subgroup $\mathbb Z_2$ inverts $\mathbb Z_p\times \mathbb Z_q$. Hence
\[
G\cong \mathbb Z_r\times(\mathbb Z_2\ltimes \mathbb{Z}_{pq})
\cong \mathbb Z_r\times D_{2pq},
\]
so Lemma~\ref{lemma 5.8} applies.

\item Assume $|a|\in\{rp,rq\}$. Without loss of generality, let $|a|=rp$.
Then $\mathbb Z_r$ centralizes $\mathbb Z_p$. 
Since $G'\cap Z(G)=\{e\}$, the subgroup $\mathbb Z_2$ acts non-trivially on $\mathbb Z_p$, and hence inverts $\mathbb Z_p$. Therefore $|b|\in\{2,2q,2r\}$.
We note that $|b|=2qr$ is already ruled out.
If $|b|=2q$, then Corollary~\ref{corollary 5.2} applies.
So assume $|b|\in\{2,2r\}$. If $\gcd(rp,r-1)=1$, then Lemma~\ref{lemma 5.16.} applies. Hence we may assume $\gcd(p,r-1)\neq 1$, so $r\equiv 1\pmod p$.

Consider
\[
\widecheck G=G/\mathbb Z_q\cong \mathbb Z_r\times D_{2p}
\cong (\mathbb Z_2\times \mathbb Z_r)\ltimes \mathbb Z_p.
\]
After conjugation, we may assume $\widecheck a=a_r\gamma_p\,\mathbb Z_q,~
\widecheck b=a_2a_r^j\,\mathbb Z_q$
for some $j$ satisfying $0\le j\le r-1$, where $\langle a_r\rangle=\mathbb Z_r$ and $\langle \gamma_p\rangle=\mathbb Z_p$.
Since $\widecheck b$ centralizes $a_r\mathbb Z_q$ and inverts $\gamma_p\mathbb Z_q$, we have
$\widecheck b^{-1}\widecheck a\,\widecheck b=\widecheck a^{k_0}$,
whenever $k_0\equiv 1\pmod r$ and $k_0\equiv -1\pmod p$.
Because $r\equiv 1\pmod p$, we may choose $k_0=r(p-2)+1$.
Let $n=rp$ and $k=n-k_0$.

%\textcolor{red}{FL: The original formulation was `there is $k_0$ such that\dots' from which we cannot conclude that $k_0 = r(p-2)+1$ works. Please check whether you agree with the way I fixed it.}

%{\color{blue} Looks good.}

\begin{subcase}
Assume $p>5$. We know $\widecheck a=a\mathbb Z_q$, and 
$\widecheck b=b\mathbb Z_q$.
Since
\[
|\widecheck G|=2rp=2n,
\qquad
|\widecheck a|=n,
\qquad
\widecheck b\notin \langle \widecheck a\rangle,
\qquad
\widecheck b^{-1}\widecheck a\widecheck b=\widecheck a^{k_0},
\]
\Cref{prop:gp-cycles} applies. Hence $\Cay(\widecheck G;\widecheck S)$ contains
hamiltonian cycles
\[
C_0,
\qquad
C_1=C_0\triangle D_1,
\qquad
C_2=C_0\triangle D_2,
\]
where $C_0$ is the cycle in \eqref{eq:gp-C0} and $D_t$ is the $8$-cycle in
\eqref{eq:gp-Ds}.

We now compute the voltage with respect to the normal subgroup $\mathbb Z_q$. For simplicity let
$\gamma=\gamma_q$ and write
$a_r\gamma a_r^{-1}=\gamma^\tau$.
If $\tau=1$, then $\mathbb Z_r$ centralizes both $\mathbb Z_p$ and $\mathbb Z_q$.
Since $\mathbb Z_2$ inverts $\mathbb Z_p$ and $G'\cap Z(G)=\{e\}$, the subgroup
$\mathbb Z_2$ must also act non-trivially on $\mathbb Z_q$, hence it inverts
$\mathbb Z_q$. Therefore $G\cong D_{2pq}\times \mathbb Z_r$,
so \Cref{lemma 5.8} applies. Hence we may assume $\tau\neq 1$. Since $r$ is prime
and $\tau^r=1$ i.e. $\tau^r \equiv 1 \mod q$, the element $\tau$ has order $r$ modulo $q$.
We note that by repeatedly multiplying the identity
$a_r\gamma a_r^{-1}=\gamma^\tau$
on the left by $a_r$ and on the right by $a_r^{-1}$, we obtain
$a_r^i\gamma a_r^{-i}=\gamma^{\tau^i}$
for every integer $i\ge 1$.

%\textcolor{red}{FL: there seems to be some detail missing here - I assume $\tau^r=1$ means $\tau \equiv 1 \mod q$? It might also be worth including a short justification along the lines of `note that $a_r^i \gamma a_r^{-i} = \gamma^{\tau^i}$' here.}
%{\color{blue} we added an explanation.}

After conjugating by an element of $\mathbb Z_q$, we may assume, modulo
$\mathbb Z_p$, that $a\equiv a_r$.
After simultaneously conjugating $a$ and $b$ by an element of
$\mathbb Z_q$, we may assume that $a\mathbb Z_p=a_r\mathbb Z_p$.
Since $G/\mathbb Z_p\cong
(\mathbb Z_2\times\mathbb Z_r)\ltimes\mathbb Z_q$
and $b\mathbb Z_q=a_2a_r^j\mathbb Z_q$, there is some
$u\in\mathbb Z/q\mathbb Z$ such that $b\mathbb Z_p=a_2a_r^j\gamma^u\mathbb Z_p$.
We must have $u\not\equiv0\pmod q$, because otherwise the images of
$a$ and $b$ would generate no element of the normal subgroup
$\mathbb Z_q$ in $G/\mathbb Z_p$, contradicting
$\langle a\mathbb Z_p,b\mathbb Z_p\rangle=G/\mathbb Z_p$.
Since $q$ is prime, $\gamma^u$ is a generator of $\mathbb Z_q$.
Replacing $\gamma$ by $\gamma^u$, we may therefore write
$b\mathbb Z_p=a_2a_r^j\gamma\mathbb Z_p$ and so $b\equiv a_2a_r^j\gamma \pmod {\mathbb Z_p}$.
Also write
\[
a_2\gamma a_2^{-1}=\gamma^\varepsilon,
\qquad \varepsilon\in\{1,-1\}.
\]
A direct multiplication, using
\[
a_r^x\gamma a_r^{-x}=\gamma^{\tau^x},
\]
gives
\[
\operatorname{Volt}(C_0)\equiv \gamma^{\varepsilon\tau^j V_0}\pmod {\mathbb Z_p},
\]
where
\[
V_0=-\tau^3+2\tau-2+2\tau^{-1}-\tau^{-3}.
\]
Equivalently, we obtain  $V_0=-\tau^{-3}(\tau-1)^2
\bigl(\tau^4+2\tau^3+\tau^2+2\tau+1\bigr)$.
If $V_0\not\equiv0\pmod q$, then the voltage generates $\mathbb Z_q$, so the
\Cref{FGL} applies. Thus assume $V_0\equiv0\pmod q$.

For $t=1,2$, we use the hamiltonian cycle $C_t=C_0\triangle D_t$
given by \Cref{prop:gp-cycles}. A direct voltage calculation gives
\[
\operatorname{Volt}(C_t)\equiv \gamma^{\varepsilon\tau^j V_t}\pmod {\mathbb Z_p}.
\]
If $p\ge 11$, then
\[
V_t=
V_0+\tau^{tp-2}(1+2\tau-\tau^2)-2\tau^2,
\qquad t=1,2.
\]
If $p=7$, then
\[
V_t=
V_0+2\tau^3-\tau^{7t-2}(1+\tau^2),
\qquad t=1,2.
\]

We now show that at least one of $V_0,V_1,V_2$ is non-zero modulo $q$. Suppose, for a
contradiction, that
\[
V_0\equiv V_1\equiv V_2\equiv0\pmod q.
\]

%\textcolor{red}{FL: Am i right in assuming that all computations with $\tau$ below are mod $q$? If so, it may be worth stating it. If not, what is going on here?}

%{\color{blue} yes it is mod $q$.}

First assume $p\ge11$. From $V_1\equiv V_0 \equiv 0 \pmod{q}$ and $V_2 \equiv V_0 \equiv 0 \pmod{q}$, we obtain
\[
\tau^{p-2}(1+2\tau-\tau^2) \equiv 2\tau^2 \pmod{q}
\]
and
\[
\tau^{2p-2}(1+2\tau-\tau^2)\equiv 2\tau^2 \pmod{q}.
\]
The factor $1+2\tau-\tau^2$ is non-zero, because $2\tau^2\not\equiv 0 \pmod{q}$.
Dividing the two displayed equations gives $\tau^p\equiv 1 \pmod{q}$.
But $\tau$ has order $r$, and $r\equiv1\pmod p$, so $r>p$. This is impossible.

It remains to consider $p=7$. Then $V_1\equiv V_0\equiv 0 \pmod{q}$ and $V_2\equiv V_0\equiv 0\pmod{q}$ imply
\[
\tau^5(1+\tau^2)\equiv 2\tau^3 \pmod{q}
\]
and
\[
\tau^{12}(1+\tau^2)\equiv 2\tau^3 \pmod{q}.
\]
Again $1+\tau^2\neq0$, so division gives $\tau^7=1$.
But $\tau$ has order $r$, and $r\equiv1\pmod7$, so $r>7$. This is impossible.

Therefore at least one of $C_0,C_1,C_2$ has non-trivial voltage modulo $\mathbb Z_q$.
Since $|\mathbb Z_q|=q$ is prime, this voltage generates $\mathbb Z_q$. Thus, \Cref{FGL} gives a hamiltonian cycle in $\Cay(G;S)$.
\end{subcase}

\begin{subcase}
Assume $p=5$. Note that $ r\geq 5 $. 
Then $\Cay(\widecheck G;\widecheck S)$ has the following
hamiltonian cycle (see~\Cref{fig:hc5}):
\[
\begin{aligned} 
C=\big(&
\widecheck a^{r-5},\widecheck b^{-1},
\widecheck a^{-(2r-4)},\widecheck b,
\widecheck a^{-1},\widecheck b^{-1},
\widecheck a^{r+1},\widecheck b,
\widecheck a^{2r-3},\widecheck b^{-1},
\widecheck a,\widecheck b,\widecheck a,\widecheck b^{-1},\\
&\qquad
\widecheck a^{-1},\widecheck b,
\widecheck a^r,\widecheck b^{-1},
\widecheck a^{r-1},\widecheck b,
\widecheck a^r,\widecheck b^{-1},
\widecheck a^{-(r-4)},\widecheck b
\big).
\end{aligned}
\]
\begin{figure}[H]
    \centering
    \includegraphics[scale=0.8]{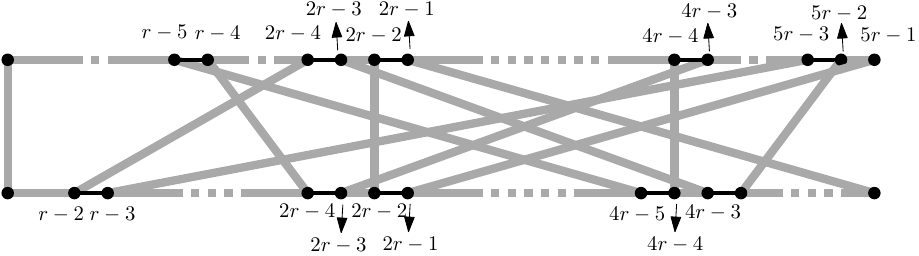}
    \caption{Hamiltonian cycle.}
    \label{fig:hc5}
\end{figure}

%\textcolor{red}{FL: I was worried about $r=3$ for a second, may be worth reminding the reader that this is already ruled out.}

%{\color{blue} Done}

We compute the voltage with respect to the normal subgroup $\mathbb Z_q$. For simplicity let
$\gamma=\gamma_q$, and write
$a_r\gamma a_r^{-1}=\gamma^\tau$.
If $\tau=1$, then $\mathbb Z_r$ centralizes both $\mathbb Z_p$ and $\mathbb Z_q$.
Since $\mathbb Z_2$ inverts $\mathbb Z_p$ and $G'\cap Z(G)=\{e\}$, the subgroup
$\mathbb Z_2$ must act non-trivially on $\mathbb Z_q$, hence it inverts $\mathbb Z_q$.
Therefore
$G\cong D_{2pq}\times \mathbb Z_r$,
so \Cref{lemma 5.8} applies. Hence we may assume $\tau\neq1$. Since $r$ is prime and
$\tau^r=1$, the element $\tau$ has order $r$ modulo $q$.

After conjugating by an element of $\mathbb Z_q$, we may assume, modulo $\mathbb Z_p$,
that $a\equiv a_r$.
Since $\langle a,b\rangle=G$, the $\mathbb Z_q$-component of $b$ is non-trivial. Replacing
$\gamma$ by a suitable generator of $\mathbb Z_q$, write
\[
b\equiv a_2a_r^j\gamma \pmod{\mathbb Z_p}.
\]
Also write
\[
a_2\gamma a_2^{-1}=\gamma^\varepsilon,
\qquad \varepsilon\in\{1,-1\}.
\]

A direct multiplication in $G/\mathbb Z_p$, using
\[
a_r^x\gamma a_r^{-x}=\gamma^{\tau^x},
\]
gives
\[
\operatorname{Volt}(C)
\equiv
\gamma^{\ell}
\pmod{\mathbb Z_p},
\]
where
\[
\ell=-\tau^{-5}+2\tau^{-1}-2\tau^{-2}-\tau^{-4}+\tau^{-3}+1.
\]
If $\ell\not\equiv0\pmod q$, then the voltage generates $\mathbb Z_q$, so
\Cref{FGL} applies. Thus assume
$\ell\equiv0\pmod q$.
Multiplying by $\tau^5$ gives
\[
0\equiv \tau^5+2\tau^4-2\tau^3+\tau^2-\tau-1
=(\tau-1)(\tau^4+3\tau^3+\tau^2+2\tau+1)\pmod q.
\]
Since $\tau\neq1$, we get
\begin{equation}\label{eq:p5-plus}
\tau^4+3\tau^3+\tau^2+2\tau+1\equiv0\pmod q.
\end{equation}

We can replace $ \tau $ with $ \tau^{-1} $ in \Cref{eq:p5-plus} by replacing $ \{\widecheck a, \widecheck b\} $ with $ \{\widecheck a^{-1},\widecheck b^{-1}\} $. Hence, we have 
\begin{equation}
\tau^{-4}+3\tau^{-3}+\tau^{-2}+2\tau^{-1}+1\equiv0\pmod q.
\end{equation}
Multiplying by $ \tau^4 $, we have
\begin{equation}\label{eq:p5}
\tau^{4}+2\tau^{3}+\tau^{2}+3\tau+1\equiv0\pmod q.
\end{equation}
By subtracting \Cref{eq:p5} from \Cref{eq:p5-plus}, we get $ \tau(\tau^2-1)\equiv 0 \pmod{q} $, which is a contradiction.

\end{subcase}
\end{itemize}
\end{case}

\begin{case}\label{case:2pqr-2}
Assume $ |\overline a|=2r $ and $ |\overline b|=2 $. Then
$\overline b=\overline a^r$, so $b=a^r\gamma$ for some $\gamma\in G'$. Since
\[
G=\langle a,b\rangle=\langle a,a^r\gamma\rangle=\langle a,\gamma\rangle,
\]
the element $\gamma$ must generate $G'$. By
\Cref{Hall Theorem}(\ref{Hall Theorem4}), we have $|a|=2r$ and
\[
a\gamma a^{-1}=\gamma^\tau,\qquad
\tau^{2r}\equiv1\pmod{pq},\qquad
\gcd(\tau-1,pq)=1.
\]
In particular, $\tau\not\equiv1\pmod p$ and $\tau\not\equiv1\pmod q$.
We have
\[
C_1=(\overline a^{r-1},\overline b,\overline a^{-(r-1)},\overline b)
\]
as a hamiltonian cycle in $\Cay(\overline G;\overline S)$. Its voltage is
\begin{align*}
    \operatorname{Volt}(C_1)
    &=a^{r-1}ba^{-(r-1)}b \\
    &=a^{r-1}a^r\gamma a^{-(r-1)}a^r\gamma \\
    &=a^{-1}\gamma a\gamma
     =\gamma^{\tau^{-1}+1}.
\end{align*}
We may assume $\gcd(\tau^{-1}+1,pq)\neq1$, for otherwise \Cref{FGL} applies.
Without loss of generality, let $\tau^{-1}\equiv -1\pmod q$, so
$\tau\equiv -1\pmod q$. We may also assume $\tau\not\equiv -1\pmod p$, for otherwise
$G\cong D_{2pq}\times\mathbb Z_r$, so \Cref{lemma 5.8} applies.
Next we consider $\widehat G=G/\mathbb Z_p=\mathbb Z_{2r}\ltimes \mathbb Z_q$.

%\textcolor{red}{FL: as far as I can tell this is the first time we use $\mathcal C$ for cyclic groups, perhaps it would be worth changing it to $(\mathbb Z_2 \times \mathbb Z_r) \ltimes \mathbb Z_q$ for consistency; it would be one less notation for the reader to keep in mind.}
%{\color{blue} Done!}
Since $|a|=2r$, we have $|\widehat a|=2r$. The image of $\gamma$ in
$G/\mathbb Z_p$ is a generator of $\mathbb Z_q$, which we denote by $\gq$. Hence we have $\widehat b=\widehat a^r\gq$.
Since $\tau\equiv -1\pmod q$, the subgroup $\mathbb Z_r$ centralizes
$\mathbb Z_q$, while $\mathbb Z_2$ inverts $\mathbb Z_q$. 
Therefore we obtain $\widehat G\cong D_{2q}\times\mathbb Z_r$.

\begin{subcase}
If $q>r$, then
\[
C_2=
\left(
 (\widehat a^{2r-1},\widehat b,\widehat a^{-(2r-1)},\widehat b)^{(q-r)/2},
 (\widehat a^{2r-1},\widehat b)^r
\right)
\]
is a hamiltonian cycle in $\Cay(\widehat G;\widehat S)$. The picture in \Cref{Fig1}
shows the hamiltonian cycle when $q=7$ and $r=3$.
\begin{figure}[H]
    \centering
    \includegraphics[scale=0.8]{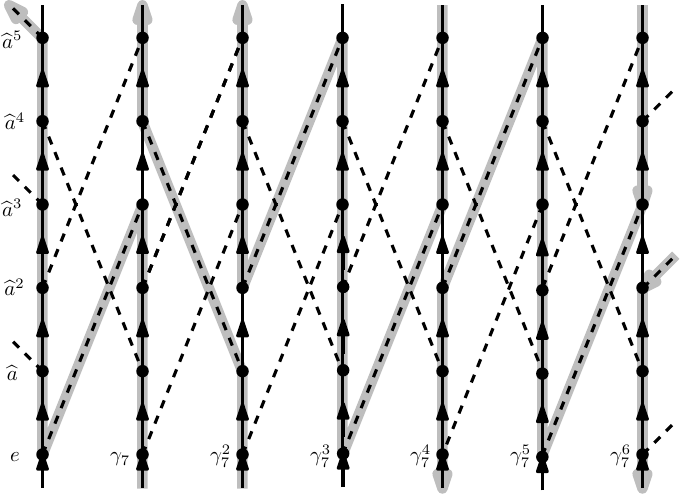}
    \caption{The union of the thick gray edges forms the hamiltonian cycle $C_2$ in $\Cay(\widehat{G}; \widehat{S})$.}
    \label{Fig1}
\end{figure}

If in $C_2$ we change one occurrence of
$(\widehat a^{2r-1},\widehat b,\widehat a^{-(2r-1)},\widehat b)$ to
$(\widehat a^{-(2r-1)},\widehat b,\widehat a^{2r-1},\widehat b)$, we obtain another hamiltonian cycle. Note that
$$a^{2r-1}ba^{-(2r-1)}b =a^{2r-1}\cdot a^r\gamma\cdot a^{-(2r-1)}\cdot a^r\gamma =a^{r-1}\gamma a^{-(r-1)}\gamma=\gamma^{\tau^{r-1}+1},$$
and
$$a^{-(2r-1)}ba^{2r-1}b =a^{-(2r-1)}\cdot a^r\gamma\cdot a^{2r-1}\cdot a^r\gamma =a^{-(r-1)}\gamma a^{r-1}\gamma=\gamma^{\tau^{-(r-1)}+1}.$$
We may assume
$\tau^{r-1}+1\equiv\tau^{-(r-1)}+1\pmod p$, for otherwise the two hamiltonian cycles
have different voltages, so one has non-trivial voltage and \Cref{FGL} applies. Hence
$\tau^{r-1}\equiv\tau^{-(r-1)}\pmod p$, so
$\tau^{2r-2}\equiv1\pmod p$. Since $\tau^{2r}\equiv1\pmod p$, this implies
$\tau^2\equiv1\pmod p$, which is impossible because
$\tau\not\equiv1\pmod p$ and $\tau\not\equiv -1\pmod p$.
\end{subcase}

\begin{subcase}
If $r>q$, then
\[
C_3=\left(
(\widehat a^r,\widehat b,\widehat a^{-r},\widehat b)^{(q-1)/2},
\widehat a^{-(r-1)},\widehat b,
(\widehat a^{r-2},\widehat b,\widehat a^{-(r-2)},\widehat b)^{(q-1)/2},
\widehat a^{r-1},\widehat b
\right)
\]
is a hamiltonian cycle in $\Cay(\widehat G;\widehat S)$. The picture in \Cref{Fig2}
shows the hamiltonian cycle when $q=3$ and $r=7$.
\begin{figure}[H]
    \centering
    \includegraphics[scale=1]{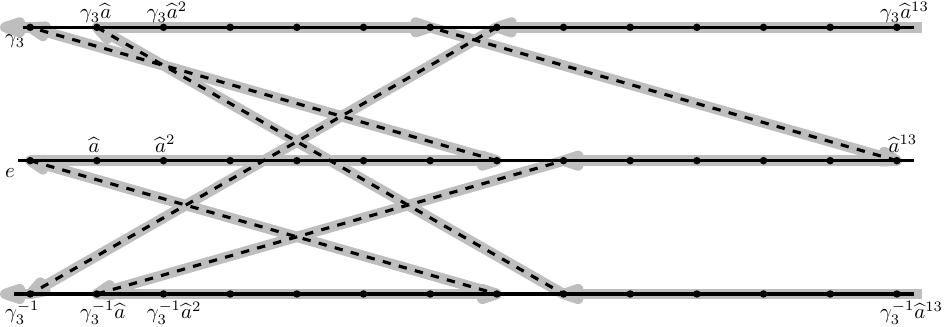}
    \caption{The union of the thick gray edges forms the hamiltonian cycle $C_3$ in $\Cay(\widehat{G}; \widehat{S})$.}
    \label{Fig2}
\end{figure}
Now we calculate its voltage:
\begin{align*}
\operatorname{Volt}(C_3)
&=(a^rba^{-r}b)^{(q-1)/2}a^{-(r-1)}b
  (a^{r-2}ba^{-(r-2)}b)^{(q-1)/2}a^{r-1}b \\
&=(a^{r}\cdot a^r\gamma\cdot a^{-r}\cdot a^r\gamma)^{(q-1)/2}
  \cdot a^{-(r-1)}\cdot a^r\gamma \\
&\qquad\qquad\cdot
  (a^{r-2}\cdot a^r\gamma\cdot a^{-(r-2)}\cdot a^r\gamma)^{(q-1)/2}
  \cdot a^{r-1}\cdot a^r\gamma \\
&=\gamma^{q-1}a\gamma(a^{-2}\gamma a^2\gamma)^{(q-1)/2}a^{-1}\gamma \\
&=\gamma^{q-1}a\gamma(\gamma^{\tau^{-2}+1})^{(q-1)/2}a^{-1}\gamma \\
&=\gamma^{(q-1)+\tau((\tau^{-2}+1)(q-1)/2+1)+1}.
\end{align*}
We may assume this does not generate $\mathbb Z_p$, for otherwise \Cref{FGL} applies. So
\[
(q-1)+\tau((\tau^{-2}+1)(q-1)/2+1)+1\equiv0\pmod p,
\]
which implies
\begin{equation}\label{eq.1.case.2}
\tau(q+1)+2q+\tau^{-1}(q-1)\equiv0\pmod p.
\end{equation}
Replacing $a$ by $a^{-1}$ in $C_3$ replaces $\tau$ by $\tau^{-1}$, so we also have
\begin{equation}\label{eq.2.case.2}
\tau(q-1)+2q+\tau^{-1}(q+1)\equiv0\pmod p.
\end{equation}
Subtracting \Cref{eq.2.case.2} from \Cref{eq.1.case.2} gives
$\tau^2\equiv1\pmod p$, which is impossible.
\end{subcase}
\end{case}

\begin{case}\label{case:2pqr-3}
Assume $ |\overline a|=|\overline b|=2r $. Then
$\overline b=\overline a^m$ for some integer $m$ with $\gcd(m,2r)=1$. Replacing
$a$ by $a^{-1}$, if necessary, we may assume $1\le m\le r-1$. In particular,
$m$ is odd. Since $\overline b=\overline a^m$, we can write $b=a^m\gamma$ with
$\gamma\in G'$. The equality $G=\langle a,b\rangle=\langle a,\gamma\rangle$ forces
$\gamma$ to generate $G'$. By
\Cref{Hall Theorem}(\ref{Hall Theorem4}), we have
\[
a\gamma a^{-1}=\gamma^\tau,\qquad
\tau^{2r}\equiv1\pmod{pq},\qquad
\gcd(\tau-1,pq)=1.
\]
Consider $\overline G=\mathcal C_{2r}$. We have
$C_4=(\overline b,\overline a^{-(m-1)},\overline b,\overline a^{2r-m-1})$
as a hamiltonian cycle in $\Cay(\overline G;\overline S)$. Its voltage is
\begin{align*}
\operatorname{Volt}(C_4)
&=ba^{-(m-1)}ba^{2r-m-1} \\
&=a^m\gamma\cdot a^{-m+1}\cdot a^m\gamma\cdot a^{2r-m-1} \\
&=\gamma^{\tau^m+\tau^{m+1}}
 =\gamma^{\tau^m(1+\tau)}.
\end{align*}
We may assume this does not generate $G'$, for otherwise \Cref{FGL} applies. Thus
$\gcd(1+\tau,pq)\neq1$. Without loss of generality, assume
$\tau\equiv -1\pmod q$.

\begin{subcase}
Assume $m>3$. Since $m$ is odd and $m<r$, we have $m\le r-2$. Then
\[
C_5=(\overline b^{-2},\overline a^{-2},\overline b,\overline a,
\overline b,\overline a^{-(m-2)},\overline b^{-1},\overline a^{m-4},
\overline b^{-1},\overline a^{-(2r-2m-3)})
\]
is a hamiltonian cycle in $\Cay(\overline G;\overline S)$.
\begin{figure}[H]
    \centering
    \includegraphics[width=0.5\linewidth]{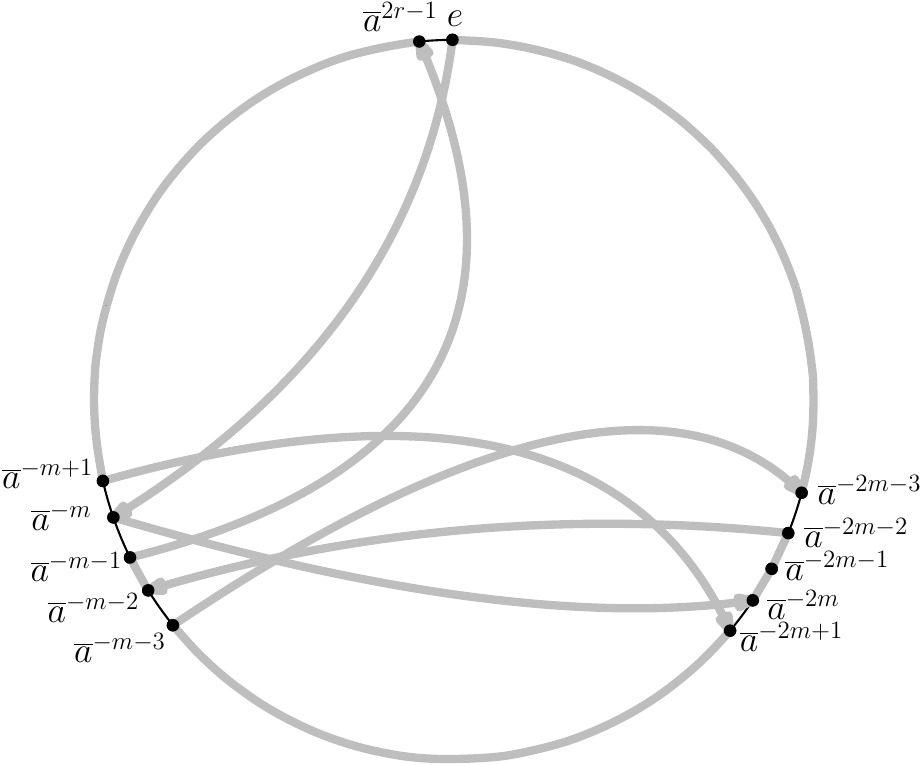}
    \caption{The union of the thick gray edges forms the hamiltonian cycle $C_5$ in $\Cay(\overline{G}; \overline{S})$.}
    \label{fig:case3-c5}
\end{figure}
Now we calculate its voltage:
\begin{align*}
\operatorname{Volt}(C_5)
&=b^{-2}a^{-2}baba^{-(m-2)}b^{-1}a^{m-4}b^{-1}a^{-(2r-2m-3)} \\
&=\gamma^{-1}a^{-m}\gamma^{-1}a^{-m}a^{-2}a^m\gamma a a^m\gamma
  a^{-m+2}\gamma^{-1}a^{-m}a^{m-4}\gamma^{-1}a^{-m}a^{-2r+2m+3} \\
&=\gamma^{-1}a^{-m}\gamma^{-1}a^{-2}\gamma a^{m+1}\gamma
  a^{-m+2}\gamma^{-1}a^{-4}\gamma^{-1}a^{m+3} \\
&=\gamma^{-1-\tau^{-m}+\tau^{-m-2}+\tau^{-1}-\tau^{-m+1}-\tau^{-m-3}} \\
&=\gamma^{-1+\tau^{-1}-\tau^{-m+1}-\tau^{-m}+\tau^{-m-2}-\tau^{-m-3}}.
\end{align*}
We may assume $\operatorname{Volt}(C_5)$ does not generate
$G'=\mathbb Z_p\times\mathbb Z_q$, for otherwise \Cref{FGL} applies. Since
$\tau\equiv -1\pmod q$ and $m$ is odd, the exponent of $\gamma$ in
$\operatorname{Volt}(C_5)$ is congruent to
\[
-1-1-1+1-1-1=-4\pmod q.
\]
Because $q\ge5$, this is non-zero modulo $q$. Hence the voltage has non-trivial
$\mathbb Z_q$-component, so the only remaining possibility is that its
$\mathbb Z_p$-component is trivial. Thus
\begin{equation}\label{7.1B}
0\equiv -1+\tau^{-1}-\tau^{-m+1}-\tau^{-m}+\tau^{-m-2}-\tau^{-m-3}\pmod p.
\end{equation}
Multiplying by $-\tau^{m+3}$, we get
\begin{equation}\label{7.1C}
0\equiv \tau^{m+3}-\tau^{m+2}+\tau^4+\tau^3-\tau+1\pmod p.
\end{equation}
Replacing $\{\overline a,\overline b\}$ by
$\{\overline a^{-1},\overline b^{-1}\}$ replaces $\tau$ by $\tau^{-1}$. Applying the
same argument gives
\begin{equation}\label{7.1D}
0\equiv -\tau^{m+3}+\tau^{m+2}-\tau^m-\tau^{m-1}+\tau-1\pmod p.
\end{equation}
Adding \Cref{7.1C} and \Cref{7.1D} yields
\[
0\equiv -\tau^m-\tau^{m-1}+\tau^4+\tau^3
  =\tau^3(\tau+1)(1-\tau^{m-4})\pmod p.
\]
If $\tau\equiv -1\pmod p$, then, since also $\tau\equiv -1\pmod q$, the quotient
$\mathcal C_{2r}$ inverts $\mathbb Z_{pq}$. Hence $\mathbb Z_r$ centralizes $G'$,
so $G\cong D_{2pq}\times\mathbb Z_r$, and \Cref{lemma 5.8} applies. Therefore we may
assume $\tau^{m-4}\equiv1\pmod p$. But $\tau^{2r}\equiv1\pmod p$, so the order of
$\tau$ modulo $p$ divides $\gcd(m-4,2r)$. Since $m$ is odd and $3<m<r$, we have
$\gcd(m-4,2r)=1$. This contradicts $\gcd(\tau-1,pq)=1$.
\end{subcase}

\begin{subcase}
Assume $m\le3$. Since $m$ is odd, we have $m\in\{1,3\}$.

First assume $m=1$. Then $\overline b=\overline a$ and $b=a\gamma$. The quotient
cycle $\overline C_6=(\overline a^{2r-1},\overline b)$
is a hamiltonian cycle in $\Cay(\overline G;\overline S)$. Indeed, it visits
\[
e,\overline a,\overline a^2,\ldots,\overline a^{2r-1},e.
\]
Its lift is $C_6=(a^{2r-1},b)$, and its voltage is
\[
\operatorname{Volt}(C_6)=a^{2r-1}b=a^{2r-1}a\gamma=\gamma.
\]
Since $\gamma$ generates $G'$, \Cref{FGL} applies.

Now assume $m=3$. Thus $\overline b=\overline a^3$ and $b=a^3\gamma$. We claim that
\[
\overline C_7=(\overline a,\overline b^{-1},\overline a,\overline b,
\overline a^{2r-5},\overline b)
\]
is a hamiltonian cycle in $\Cay(\overline G;\overline S)$. Indeed, since
$\overline b=\overline a^3$, the successive vertices of this walk are
\[
e,\ \overline a,\ \overline a^{-2},\ \overline a^{-1},\
\overline a^2,\ \overline a^3,\ldots,\overline a^{2r-3},\ e.
\]
These are all $2r$ vertices of $\overline G=\langle\overline a\rangle$, before the final
return to $e$. Therefore $\overline C_7$ is a hamiltonian cycle.
Let
\[
C_7=(a,b^{-1},a,b,a^{2r-5},b)
\]
be the corresponding lift to $G$. Since $b=a^3\gamma$, we have
$b^{-1}=\gamma^{-1}a^{-3}$. Therefore
\begin{align*}
\operatorname{Volt}(C_7)
&=ab^{-1}aba^{2r-5}b \\
&=a\gamma^{-1}a^{-3}a\,a^3\gamma a^{2r-5}a^3\gamma \\
&=a\gamma^{-1}a\gamma a^{2r-2}\gamma \\
&=\gamma^{-\tau+\tau^2+\tau^{2r}} \\
&=\gamma^{\tau^2-\tau+1},
\end{align*}
where the last equality uses $\tau^{2r}\equiv1\pmod{pq}$.

We show that $\tau^2-\tau+1$ is relatively prime to $pq$. Since
$\tau\equiv -1\pmod q$, we have
\[
\tau^2-\tau+1 = (-1)^2-(-1)+1=3\not\equiv0\pmod q,
\]
because $q\ge5$.
It remains to show that $\tau^2-\tau+1\not\equiv0\pmod p$. Suppose, for a contradiction, that
\[
\tau^2-\tau+1\equiv0\pmod p.
\]
Then
\[
\tau^3+1=(\tau+1)(\tau^2-\tau+1)\equiv0\pmod p.
\]
Also $\tau\not\equiv -1\pmod p$, because substituting $\tau=-1$ into
$\tau^2-\tau+1$ gives $3\not\equiv0\pmod p$. Hence $\tau^3\equiv-1\pmod p$, so
$\tau$ has order $6$ modulo $p$. But $\tau^{2r}\equiv1\pmod p$, so the order of
$\tau$ modulo $p$ must divide $2r$. This is impossible, because $r\ge5$ is an odd
prime, so $6\nmid2r$.

Therefore $\tau^2-\tau+1\not\equiv0\pmod p$ and $\tau^2-\tau+1\not\equiv0\pmod q$. Since
$G'=\langle\gamma\rangle$ has order $pq$, it follows that $\gamma^{\tau^2-\tau+1}$ generates
$G'$. Hence $\operatorname{Volt}(C_7)$ generates $G'$, and \Cref{FGL} applies.\qedhere
\end{subcase}
\end{case}
\end{proof}

\begin{proof}[Proof of Theorem~\ref{main_1}]
Let $S_0\subseteq S$ be a minimal generating subset of $G$. If $|S_0|\ge3$, then
Theorem~\ref{pqrs-3} gives a hamiltonian cycle in $\Cay(G;S_0)$. Since
$\Cay(G;S_0)$ is a spanning subgraph of $\Cay(G;S)$, this is also a hamiltonian cycle
in $\Cay(G;S)$.

Thus we may assume $|S_0|\le2$. If all four prime divisors of $|G|$ are odd, then
Theorem~\ref{theorem 1.2}\eqref{pqrs} applies. Hence we may assume that one of the
prime divisors is $2$, so $|G|=2pqr$. If one of $p,q,r$ is $3$, then
$|G|=6uv$ for two distinct primes $u,v$, and
Theorem~\ref{theorem 1.2}\eqref{theorem 1.2.2} applies. Therefore the only remaining
case is $|G|=2pqr$ with $p,q,r\ge5$. Applying Theorem~\ref{2pqr} to the
generating set $S_0$, we obtain a hamiltonian cycle in $\Cay(G;S_0)$. Since
$\Cay(G;S_0)$ is a spanning subgraph of $\Cay(G;S)$, this is also a hamiltonian
cycle in $\Cay(G;S)$. This completes the proof.
\end{proof}

\section*{Acknowledgments}
We thank Dave Witte Morris for a discussion during the preparation of this paper.


\begin{thebibliography}{10}

\bibitem{8pq}
Fateme Abedi, Dave~Witte Morris, Javanshir Rezaee, and Mohammad~Reza Salarian.
\newblock {C}ayley graphs of order {$8pq$} are {Hamiltonian}.
\newblock {\em Contributions to Discrete Mathematics}, 20(2):311--336, 2025.
\newblock \href {http://arxiv.org/abs/2304.03348} {\path{arXiv:2304.03348}},
  \href {https://doi.org/10.55016/ojs/cdm.v20i2.77376}
  {\path{doi:10.55016/ojs/cdm.v20i2.77376}}.

\bibitem{babai1979problem}
L{\'a}szl{\'o} Babai.
\newblock Problem 17, unsolved problems.
\newblock Summer Research Workshop in Algebraic Combinatorics, Simon Fraser
  University, 1979.

\bibitem{babai1995automorphism}
L{\'a}szl{\'o} Babai.
\newblock Automorphism groups, isomorphism, reconstruction.
\newblock In Ronald~L. Graham, Martin Gr{\"o}tschel, and L{\'a}szl{\'o}
  Lov{\'a}sz, editors, {\em Handbook of Combinatorics}, volume~2, chapter~27,
  pages 1447--1540. North-Holland, Amsterdam, 1995.

\bibitem{bermond1978hamiltonian}
Jean-Claude Bermond.
\newblock {Hamiltonian} graphs.
\newblock In Lowell~W. Beineke and Robin~J. Wilson, editors, {\em Selected
  Topics in Graph Theory}, pages 127--167. Academic Press, London, 1978.

\bibitem{BryantHerkeMaenhautWebb2018}
Darryn Bryant, Sarada Herke, Barbara Maenhaut, and Bridget~S. Webb.
\newblock On {Hamilton} decompositions of infinite circulant graphs.
\newblock {\em Journal of Graph Theory}, 88(3):434--448, 2018.
\newblock \href {https://doi.org/10.1002/jgt.22223}
  {\path{doi:10.1002/jgt.22223}}.

\bibitem{Quimpo}
Chuan~Chong Chen and Norman~F. Quimpo.
\newblock On some classes of {Hamiltonian} graphs.
\newblock {\em Southeast Asian Bulletin of Mathematics}, (Special
  Issue):252--258, 1979.

\bibitem{Fifth}
Stephen~J. Curran, Dave~Witte Morris, and Joy Morris.
\newblock {C}ayley graphs of order {$16p$} are {Hamiltonian}.
\newblock {\em Ars Mathematica Contemporanea}, 5(2):185--211, 2012.
\newblock \href {https://doi.org/10.26493/1855-3974.207.8e0}
  {\path{doi:10.26493/1855-3974.207.8e0}}.

\bibitem{darijani2025arc}
Iren Darijani, Babak Miraftab, and Dave~Witte Morris.
\newblock Arc-disjoint {Hamiltonian} paths in cartesian products of directed
  cycles.
\newblock {\em Ars Mathematica Contemporanea}, 25(2):Article P2.10, 2025.
\newblock \href {https://doi.org/10.26493/1855-3974.3047.c2d}
  {\path{doi:10.26493/1855-3974.3047.c2d}}.

\bibitem{ErdeLehner2022}
Joshua Erde and Florian Lehner.
\newblock {Hamiltonian} decompositions of 4-regular {C}ayley graphs of infinite
  abelian groups.
\newblock {\em Journal of Graph Theory}, 101(3):559--571, 2022.
\newblock \href {https://doi.org/10.1002/jgt.22840}
  {\path{doi:10.1002/jgt.22840}}.

\bibitem{ErdeLehnerPitz2020}
Joshua Erde, Florian Lehner, and Max Pitz.
\newblock {Hamilton} decompositions of one-ended {C}ayley graphs.
\newblock {\em Journal of Combinatorial Theory, Series B}, 140:171--191, 2020.
\newblock \href {https://doi.org/10.1016/j.jctb.2019.05.005}
  {\path{doi:10.1016/j.jctb.2019.05.005}}.

\bibitem{Sixth}
Ebrahim Ghaderpour and Dave~Witte Morris.
\newblock {C}ayley graphs of order {$30p$} are {Hamiltonian}.
\newblock {\em Discrete Mathematics}, 312(24):3614--3625, 2012.
\newblock \href {https://doi.org/10.1016/j.disc.2012.08.017}
  {\path{doi:10.1016/j.disc.2012.08.017}}.

\bibitem{Godsil}
Chris Godsil and Gordon Royle.
\newblock {\em Algebraic Graph Theory}, volume 207 of {\em Graduate Texts in
  Mathematics}.
\newblock Springer, New York, 2001.
\newblock \href {https://doi.org/10.1007/978-1-4613-0163-9}
  {\path{doi:10.1007/978-1-4613-0163-9}}.

\bibitem{Hall}
Marshall Hall, Jr.
\newblock {\em The Theory of Groups}.
\newblock Macmillan, New York, 1959.

\bibitem{kutnar2009hamilton}
Klavdija Kutnar and Dragan Maru{\v{s}}i{\v{c}}.
\newblock {Hamilton} cycles and paths in vertex-transitive graphs---current
  directions.
\newblock {\em Discrete Mathematics}, 309(17):5491--5500, 2009.
\newblock \href {https://doi.org/10.1016/j.disc.2009.02.017}
  {\path{doi:10.1016/j.disc.2009.02.017}}.

\bibitem{Tenth}
Klavdija Kutnar, Dragan Maru{\v{s}}i{\v{c}}, Dave~Witte Morris, Joy Morris, and
  Primo{\v{z}} {\v{S}}parl.
\newblock {Hamiltonian} cycles in {C}ayley graphs whose order has few prime
  factors.
\newblock {\em Ars Mathematica Contemporanea}, 5(1):27--71, 2012.
\newblock \href {https://doi.org/10.26493/1855-3974.177.341}
  {\path{doi:10.26493/1855-3974.177.341}}.

\bibitem{LEHNER2026114798}
Florian Lehner, Farzad Maghsoudi, and Babak Miraftab.
\newblock {Hamiltonicity} of transitive graphs whose automorphism group has
  {$\mathbb{Z}_p$} as commutator subgroups.
\newblock {\em Discrete Mathematics}, 349(3):114798, 2026.
\newblock \href {https://doi.org/10.1016/j.disc.2025.114798}
  {\path{doi:10.1016/j.disc.2025.114798}}.

\bibitem{Farzad}
Farzad Maghsoudi.
\newblock {C}ayley graphs of order {$6pq$} and {$7pq$} are {Hamiltonian}.
\newblock {\em The Art of Discrete and Applied Mathematics}, 5(1):Paper No.
  P1.10, 2022.
\newblock 55 pages.
\newblock \href {https://doi.org/10.26493/2590-9770.1389.fa2}
  {\path{doi:10.26493/2590-9770.1389.fa2}}.

\bibitem{Farzad2}
Farzad Maghsoudi.
\newblock On {Hamiltonicity} of {C}ayley graphs of order {$pqrs$}.
\newblock {\em Australasian Journal of Combinatorics}, 84:124--166, 2022.
\newblock URL: \url{https://ajc.maths.uq.edu.au/pdf/84/ajc_v84_p124.pdf}.

\bibitem{MiraftabMorris2025}
Babak Miraftab and Dave~Witte Morris.
\newblock On vertex-transitive graphs with a unique {Hamiltonian} cycle.
\newblock {\em Journal of Graph Theory}, 108(1):65--99, 2025.
\newblock \href {https://doi.org/10.1002/jgt.23166}
  {\path{doi:10.1002/jgt.23166}}.

\bibitem{pqrs}
Dave~Witte Morris.
\newblock On {Hamiltonian} cycles in {C}ayley graphs of order {$pqrs$}.
\newblock {\em The Art of Discrete and Applied Mathematics}, 5(3):Paper No.
  3.12, 2022.
\newblock 10 pages.
\newblock \href {https://doi.org/10.26493/2590-9770.1442.cb1}
  {\path{doi:10.26493/2590-9770.1442.cb1}}.

\bibitem{Fourteen}
Dave~Witte Morris and Kirsten Wilk.
\newblock {C}ayley graphs of order {$kp$} are {Hamiltonian} for {$k < 48$}.
\newblock {\em The Art of Discrete and Applied Mathematics}, 3(2):Paper No.
  P2.02, 2020.
\newblock 17 pages.
\newblock \href {https://doi.org/10.26493/2590-9770.1250.763}
  {\path{doi:10.26493/2590-9770.1250.763}}.

\bibitem{Eighteen}
David Witte.
\newblock {C}ayley digraphs of prime-power order are {Hamiltonian}.
\newblock {\em Journal of Combinatorial Theory, Series B}, 40(1):107--112,
  1986.
\newblock \href {https://doi.org/10.1016/0095-8956(86)90068-7}
  {\path{doi:10.1016/0095-8956(86)90068-7}}.

\bibitem{Twelve}
David Witte and Joseph~A. Gallian.
\newblock A survey: {Hamiltonian} cycles in {C}ayley graphs.
\newblock {\em Discrete Mathematics}, 51(3):293--304, 1984.
\newblock \href {https://doi.org/10.1016/0012-365X(84)90010-4}
  {\path{doi:10.1016/0012-365X(84)90010-4}}.

\end{thebibliography}
\end{document}